\documentclass{amsart}

\usepackage[lite,initials,nobysame]{amsrefs}
\usepackage{amssymb}
\usepackage{hyperref}
\usepackage{url}

\DeclareMathOperator{\Q}{\mathbb{Q}}

\DeclareMathOperator{\Res}{\mathrm{Resultant}}
\DeclareMathOperator{\N}{\mathcal{N}}
\DeclareMathOperator{\Gal}{\mathrm{Gal}}
\DeclareMathOperator{\disc}{\mathrm{disc}}

\usepackage{float,array,diagbox}
\newcolumntype{C}[1]{>{\centering\let\newline\\\arraybackslash\hspace{0pt}}m{#1}}

\numberwithin{equation}{section}

\theoremstyle{plain} %% This is the default, anyway
\newtheorem{theorem}{Theorem}[section]

\newtheorem{lemma}[theorem]{Lemma}
\newtheorem{proposition}[theorem]{Proposition}

\usepackage{xcolor}

\begin{document}

\title[Galois groups of $x^{12}+ax^6+b$]{Elementary characterization of \\ the Galois groups of $x^{12}+ax^6+b$}

\author[Malcolm H. W. Chen]{Malcolm Hoong Wai Chen} \address{Department of Mathematics, University of Manchester, Oxford Road, Manchester M13 9PL, United Kingdom.} \email{malcolmhoongwai.chen@manchester.ac.uk}


 \begin{abstract}
Let $f(x)=x^{12}+ax^6+b \in \Q[x]$ be an irreducible polynomial, $g_4(x)=x^4+ax^2+b$, $g_6(x)=x^6+ax^3+b$, and let $G_4$ and $G_6$ be the Galois groups of $g_4(x)$ and $g_6(x)$, respectively. This paper provides a streamlined elementary characterization of all sixteen possible Galois groups of $f(x)$. In particular, we give explicit criteria showing that the Galois group of $f(x)$ can be uniquely determined by the pair $(G_4,G_6)$ along with testing whether at most two expressions involving $a$ and $b$ are rational squares. This contributes to a broader program to systematically study the Galois groups of $x^{2m}+ax^m+b$ for a positive integer $m$.
 \end{abstract}

\maketitle

\section{Introduction}

Let $f(x)$ be an irreducible polynomial with rational coefficients. An interesting problem is to characterize the Galois group of $f(x)$ using its coefficients. A classical example is that if $f(x)=x^4+ax^2+b$, then its Galois group can be determined by testing whether $b$ and $b(a^2-4b)$ are rational squares \cite{kappewarren}.

In recent years much attention has been devoted to \emph{power compositional polynomials}, that is, irreducible polynomials of the form $g(x^m)$ for some monic $g(x) \in \Q[x]$ and positive integer $m \ge 2$. If $g(x)=x^2+ax+b$, there are complete characterizations for the cases where $m=2$ and $m=3$ given by \cite{kappewarren} and \cite{awtreysextic}, respectively. There are also partial results when $m=4$ \cite{awtreyoctic,mypaper}, eventually leading to a complete characterization in \cite{awtreyoctic2}. Some other power compositional polynomials, namely $x^6+ax^4+bx^2+c$ \cite{awtreysextic2}, $x^8+ax^6+bx^4+ax^2+1$ \cite{awtreyoctic3}, and $x^9+ax^6+bx^3+c$ \cite{awtreynonic} have been completely characterized too. %Each of these classification involves studying the subfield defined by $g(x)$ and its Galois group. 

In this paper, we will prove a complete characterization for the Galois group $G_{12}$ of $f(x)=x^{12}+ax^6+b$ (Theorem \ref{mthm}), motivated by the observation that $f(x)$ can be expressed as a power compositional polynomial in two different ways. In particular, we have $f(x)=g_4(x^3)=g_6(x^2)$, where $g_4(x)=x^4+ax^2+b$ and $g_6(x)=x^6+ax^3+b$, respectively. To the best of our knowledge, this complete classification has not yet appeared in the literature in such a streamlined form. 

We highlight that this classification is valuable on its own, as seen by its application in the recent work \cite{jones}, which used it to completely classify polynomials of the form $x^{12}+ax^6+b$ which are monogenic. We also view this classification as a stepping stone towards a more systematic study of $x^{2m}+ax^m+b$ for composite $m$; see also \cite{interesting} for a similar analysis for prime $m$.

\subsection*{Outline}
Let $G_4$ and $G_6$ be the Galois groups of $g_4(x)$ and $g_6(x)$, respectively. In Section \ref{prelim}, we will recall some classical results on Galois groups, linear resolvents and factorizations of power compositional polynomials. In Section \ref{G12}, we work towards a complete characterization of $G_{12}$. We first determine the possible candidates for the Galois group of $G_{12}$ and the extent to which $(G_4,G_6)$ uniquely identifies $G_{12}$ in Section \ref{G12a}. We then distinguish the remaining cases in Section \ref{G12b}, and finally summarize our results in Section \ref{G12d}.

\subsection*{Notation} We will be using the following notations throughout the paper: 
\begin{itemize}
\item $A_n$ : alternating group on $n$ letters, 
\item $S_n$ : symmetric group on $n$ letters, 
\item $nTj$ : $j$-th conjugacy class among transitive subgroups of $S_n$ (see \cite{butlermckay}), 
\item $\Gal(f)$ : Galois group of the polynomial $f(x)$,
\item $\disc(f)$ : discriminant of the polynomial $f(x)$,
\item $K^2$ : set of squares in a field $K$, 
\item $K^3$ : set of cubes in a field $K$.
\end{itemize}

\section{Preliminaries} \label{prelim}

We begin by recalling some important results on linear resolvents.

\begin{proposition}[\cite{soicher}] \label{soicher} 
Let $f(x) \in \Q[x]$ be an irreducible polynomial of degree $n$ and $\alpha_1,\dots,\alpha_n$ be all the roots of $f(x)$. Let $F(x_1,\dots,x_n) \in \Q[x_1,\dots,x_n]$ and let
\begin{equation*}
H=\{\sigma \in S_n : F(x_{\sigma(1)},\dots,x_{\sigma(n)}) = F(x_1,\dots,x_n)\}.
\end{equation*}
Then the resolvent polynomial of $f(x)$ corresponding to $F$ is
\begin{equation*}
R(x):=\prod_{\sigma \in S_n//H} \left(x-F(\alpha_{\sigma(1)},\dots,\alpha_{\sigma(n)})\right) \in \Q[x]
\end{equation*}
where $S_n//H$ is a complete set of right coset representatives of $H$ in $S_n$. The irreducible factors of $R(x)$ that occur with multiplicity one correspond to the orbits of the action of $Gal(f)$ on the cosets in $S_n/H$, and the Galois group for any of these irreducible factors is the image of the permutation representation of this action on its corresponding coset.
\end{proposition}

The linear resolvent corresponding to $F=x_1+x_2$ can be computed using resultants in Mathematica \cite{wmath} as the polynomial $R(x)$ satisfying the relation 
\begin{equation}
R(x)^2 = \frac{\Res_y(f(y),f(x-y))}{2^n \cdot f(x/2)}. \label{resultant}
\end{equation}

Likewise, the linear resolvent corresponding to $F=x_1 x_2$ can also be determined similarly by the relation
\begin{equation}
R(x)^2 = \frac{\Res_y(f(y),y^n f(x/y))}{\Res_y(f(y),x-y^2)}. \label{resultant2}
\end{equation}

For both cases, the list of degrees of irreducible factors of the linear resolvent $R(x)$ is equivalent to the list of orbit lengths for the action of $\Gal(f)$ on all 2-sets of $n$ letters, which can be computed in GAP \cite{gap}.

We now state a result on the irreducibility of composition of polynomials, which we will later specialize with $g(x)=x^m$ to determine the factorization patterns of power compositional polynomials we are interested in. 

\begin{proposition}[{\cite[Section 2.1]{schinzel}}] \label{capelli}
Let $f(x),g(x) \in \Q[x]$ where $f(x)$ is irreducible, and let $\alpha$ be a root of $f(x)$. Then $f(g(x))$ is reducible over $\Q$ if and only if $g(x)-\alpha$ is reducible over $\Q(\alpha)$. Furthermore, if
\begin{eqnarray*}
g(x)-\alpha = c_1 u_1(x)^{e_1} \cdots u_k(x)^{e_k} \in \Q(\alpha)[x]
\end{eqnarray*}
where $u_1(x),\dots,u_k(x)$ are distinct monic polynomials irreducible over $\Q(\alpha)$, then
\begin{eqnarray*}
f(g(x)) = c_2 \N(u_1(x))^{e_1} \cdots \N(u_k(x))^{e_k} \in \Q[x]
\end{eqnarray*}
where the norms $\N(u_1(x)),\dots,\N(u_k(x))$ are distinct monic polynomials irreducible over $\Q$.
\end{proposition}

We now recall some results on Galois groups. We identify Galois groups as transitive permutation groups up to conjugacy. Let $f(x)=g(x^m)$ be an irreducible power compositional polynomial, $K_f$ and $K_g$ be stem fields of $f(x)$ and $g(x)$, respectively, and let $H_f$ and $H_g$ be subgroups of $\Gal(f)$ corresponding to $K_f$ and $K_g$, respectively. Then $H_f \leq H_g$ is the point stabilizer of 1 under $\Gal(f)$ and the splitting field of $g(x)$ is the normal closure of $K_g$. The list of Galois groups of the normal closures of non-isomorphic intermediate subfields of $K_f$ is an invariant of $\Gal(f)$ commonly referred to the \emph{subfield content of $f(x)$}, which can be computed in GAP \cite{gap}. In particular, this list must contain $\Gal(g)$. We also recall that $\Gal(f) \leq A_n$ if and only if $\disc(f) \in \Q^2$.

Lastly, we recall known characterizations of $G_4$ and $G_6$ in the literature.

\begin{proposition}{\cite[Theorem 3]{kappewarren}} \label{g4}
Let $g_4(x)=x^4+ax^2+b \in \Q[x]$ be an irreducible polynomial. Then $G_4=\Gal(g_4)$ is
\begin{enumerate}
\item $4T1$ if $b(a^2-4b) \in \Q^2$,
\item $4T2$ if $b \in \Q^2$,
\item $4T3$ if $b \notin \Q^2$ and $b(a^2-4b) \notin \Q^2$.
\end{enumerate}
\end{proposition}

\begin{proposition}{\cite[Theorem 1.1]{cavallo}} \label{g6}
Let $g_6(x)=x^6+ax^3+b \in \Q[x]$ be an irreducible polynomial, $d=3(4b-a^2)$ and $r(x)=x^3-3bx+ab$. Then $G_6=\Gal(g_6)$ is
\begin{enumerate}
\item $6T2$ if $d \in \Q^2$ and $r(x)$ is reducible.
\item $6T1$ if $d \in \Q^2$, $b \in \Q^3$ and $r(x)$ is irreducible.
\item $6T5$ if $d \in \Q^2$, $b \notin \Q^3$ and $r(x)$ is irreducible.
\item $6T3$ if $d \notin \Q^2$ and either $b \in \Q^3$ or $r(x)$ is reducible.
\item $6T9$ if $d \notin \Q^2$, $b \notin \Q^3$ and $r(x)$ is irreducible.
\end{enumerate}
\end{proposition}

\section{Galois groups of $x^{12}+ax^6+b$} \label{G12}

\subsection{Possible Galois groups} \label{G12a}

We now establish some preliminary results required to completely classify $G_{12}$. 

\begin{lemma} \label{lem1}
Let $f(x)=x^{12}+ax^6+b \in \Q[x]$ be an irreducible polynomial. Then 
\begin{enumerate}
\item $\disc(f) \in \Q^2$ if and only if $G_4$ is $4T2$.
\item $|\Gal(f)| \leq \min\{18|G_4|,4|G_6|\}$.
\end{enumerate}
\end{lemma}

\begin{proof}
\textbf{(1)} Note that
\begin{eqnarray*}
\disc(f)=2^{12}3^{12}b^5(a^2-4b)^6 = b \big(2^6 3^6 b^2 (a^2-4b)^3\big)^2,
\end{eqnarray*}
so $\disc(f) \in \Q^2$ if and only if $b \in \Q^2$ if and only if $G_4$ is $4T2$ by Proposition \ref{g4}.
\par\textbf{(2)} Let $\pm\alpha,\pm\beta$ be the roots of $x^4+ax^2+b$ and $\omega$ be a primitive third root of unity. Then $\pm\sqrt[3]{\alpha},\pm\sqrt[3]{\alpha}\omega,\pm\sqrt[3]{\alpha}\omega^2,\pm\sqrt[3]{\beta},\pm\sqrt[3]{\beta}\omega,\pm\sqrt[3]{\beta}\omega^2$ are the roots of $f(x)$. Set $K=\Q(\alpha,\beta)$ and it follows that
\begin{equation*}
\begin{split}
|\Gal(f)| &= [\Q(\sqrt[3]{\alpha},\sqrt[3]{\beta},\omega):\Q] \\
&= [K(\sqrt[3]{\alpha},\sqrt[3]{\beta},\omega):K(\sqrt[3]{\beta},\omega)][K(\sqrt[3]{\beta},\omega):K(\omega)][K(\omega):K][K:\Q] \\
&\leq (3)(3)(2)|G_4| \\
&= 18|G_4|.
\end{split}
\end{equation*}
Now let $\gamma,\gamma\omega,\gamma\omega^2,\delta,\delta\omega,\delta\omega^2$ be the roots of $x^6+ax^3+b$, where $\omega=(-1+\sqrt{-3})/2$ is a primitive third root of unity. Then $\pm\sqrt{\gamma},\pm\sqrt{\gamma}\sqrt{\omega},\pm\sqrt{\gamma} \omega,\pm\sqrt{\delta},\pm\sqrt{\delta}\sqrt{\omega},\pm\sqrt{\delta}\omega$ are the roots of $f(x)$. Note that $\sqrt{\omega}=(1+\sqrt{-3})/2 = \omega + 1 \in \Q(\omega)$. Set $K=\Q(\gamma,\delta,\omega)$ and it follows that
\begin{equation*}
\begin{split}
|\Gal(f)| &= [\Q(\sqrt{\gamma},\sqrt{\delta},\sqrt{\omega}):\Q] \\
&= [K(\sqrt{\gamma},\sqrt{\delta},\sqrt{\omega}):K(\sqrt{\delta},\sqrt{\omega})][K(\sqrt{\delta},\sqrt{\omega}):K(\sqrt{\omega})][K(\sqrt{\omega}):K][K:\Q] \\
&\leq (2)(2)(1)|G_6| \\
&= 4|G_6|.
\end{split}
\end{equation*}
The result follows from combining both inequalities.
\end{proof}

\begin{lemma} \label{squarelem}
Let $f(x)=x^4+ax^2+b \in \Q[x]$ be an irreducible polynomial and $\theta$ be a root of $f(x)$. Then for every $r \in \Q \setminus \Q^2$, we have $r \in \Q(\theta)^2$ if and only if $r(a^2-4b)$, $r(-a+2\sqrt{b})$ or $r(-a-2\sqrt{b})$ is in $\Q^2$.
\end{lemma}

\begin{proof}
Set $\sqrt{r}=a_3 \theta^3 + a_2 \theta^2 + a_1 \theta  + a_0$ for some $a_0,a_1,a_2,a_3 \in \Q$. We claim that either $a_1=a_3=0$ or $a_0=a_2=0$. Suppose to the contrary that this is not the case. Since $\Gal(f)$ acts transitively on the roots of $f(x)$, there is a mapping $\tau \in \Gal(f)$ satisfying $\tau(\theta)=-\theta$. It follows that
\begin{eqnarray*}
\tau(\sqrt{r}) = a_2 \theta^2 + a_0 - (a_3 \theta^3 + a_1 \theta),
\end{eqnarray*}
and so $\tau(\sqrt{r}) \neq \sqrt{r}$. Similarly, $\tau(\sqrt{r})$ is another root of the minimal polynomial of $\sqrt{r}$ over $\Q$. Since $r \in \Q$, the minimal polynomial is $x^2-r$ and hence, $\tau(\sqrt{r})=-\sqrt{r}$. It follows that $a_0=a_2=0$, a contradiction.
\par\textbf{Case 1:} $a_1=a_3=0$. Then 
\begin{eqnarray*}
r=(a_2 \theta^2 + a_0)^2 = (2a_0 a_2 - a a_2^2)\theta^2 + (a_0^2-ba_2^2).
\end{eqnarray*}
The coefficient of $\theta^2$ is zero, so $a_2 \neq 0$ and $a_0 = {aa_2}/2$. It follows that $r=(a^2-4b)(a_2/2)^2$ and hence, $r(a^2-4b) \in \Q^2$.
\par\textbf{Case 2:} $a_0=a_2=0$. Then
\begin{eqnarray*}
r=(a_3 \theta^3 + a_1 \theta)^2 = (a_1^2 - 2aa_1 a_3 + a^2 a_3^2 - ba_3^2)\theta^2 + (aba_3^2-2ba_1a_3).
\end{eqnarray*}
The coefficient of $\theta^2$ is zero, so $a_3 \neq 0$ and $a_1=a_3(a\pm\sqrt{b})$. It follows that $r=(-a \mp 2\sqrt{b})(a_3\sqrt{b})^2$ and hence, either $r(-a+2\sqrt{b}) \in \Q^2$ or $r(-a-2\sqrt{b}) \in \Q^2$.
\par This proves the necessity. The sufficiency follows from construction.
\end{proof}

\begin{proposition} \label{prop2}
Let $f(x)=x^{12}+ax^6+b \in \Q[x]$ be an irreducible polynomial and $\theta$ be a root of $f(x)$. Then for every $r \in \Q \setminus \Q^2$, we have $r \in \Q(\theta)^2$ if and only if $r(a^2-4b)$, $r(-a+2\sqrt{b})$ or $r(-a-2\sqrt{b})$ is in $\Q^2$.
\end{proposition}

\begin{proof}
Suppose that $r \in \Q(\theta)^2$. Then
\begin{equation}
\begin{split}
\sqrt{r} &= a_{11} \theta^{11} + a_{10} \theta^{10} + a_9 \theta^9 + a_8 \theta^8 + a_7 \theta^7 + a_6 \theta^6 \\
&+ a_5 \theta^5 + a_4 \theta^4 + a_3 \theta^3 + a_2 \theta^2 + a_1 \theta + a_0 \label{Eq1}
\end{split}
\end{equation}
for some $a_0,a_1,\dots,a_{11} \in \Q$. Since $f(x)$ is irreducible, $\Gal(f)$ acts transitively on the roots of $f(x)$. In particular, there exist mappings $\sigma_1,\sigma_2 \in \Gal(f)$ satisfying $\sigma_1(\theta)=\theta\omega$ and $\sigma_2(\theta)=\theta\omega^2$ where $\omega$ is a primitive third root of unity. Then
\begin{equation}
\begin{split}
\sigma_1(\sqrt{r}) &= (a_{11} \theta^{11} + a_8 \theta^8 + a_5 \theta^5 + a_2 \theta^2) \omega^2 + (a_{10} \theta^{10} + a_7 \theta^7 + a_4 \theta^4 + a_1 \theta) \omega \\ &+ (a_9 \theta^9 + a_6 \theta^6 + a_3 \theta^3 + a_0), \label{Eq2}
\end{split}
\end{equation}
\begin{equation}
\begin{split}
\sigma_2(\sqrt{r}) &= (a_{11} \theta^{11} + a_8 \theta^8 + a_5 \theta^5 + a_2 \theta^2) \omega + (a_{10} \theta^{10} + a_7 \theta^7 + a_4 \theta^4 + a_1 \theta) \omega^2 \\ &+ (a_9 \theta^9 + a_6 \theta^6 + a_3 \theta^3 + a_0). \label{Eq3}
\end{split}
\end{equation}
By construction, $\sigma_1(\sqrt{r})$ and $\sigma_2(\sqrt{r})$ are roots of the minimal polynomial of $\sqrt{r}$ over $\Q$. Since $r \in \Q$, this minimal polynomial is $x^2-r$ and hence, $\sigma_1(\sqrt{r}),\sigma_2(\sqrt{r}) \in \{\pm\sqrt{r}\}$. Now if $\{\sigma_1(\sqrt{r}),\sigma_2(\sqrt{r})\}=\{\pm\sqrt{r}\}$ or $\sigma_1(\sqrt{r})=\sigma_2(\sqrt{r})=-\sqrt{r}$, taking the sum of $(\ref{Eq1})$, $(\ref{Eq2})$ and $(\ref{Eq3})$, we have $\pm\sqrt{r} = 3(a_9 \theta^9 + a_6 \theta^6 + a_3 \theta^3 + a_0)$. Substituting this into (\ref{Eq1}) we obtain a polynomial of degree less than twelve that has $\theta$ as a root, a contradiction. This implies that $\sigma_1(\sqrt{r})=\sigma_2(\sqrt{r})=\sqrt{r}$, so that by taking the sums of $(\ref{Eq1})$, $(\ref{Eq2})$ and $(\ref{Eq3})$, we have $3\sqrt{r} = 3(a_9 \theta^9 + a_6 \theta^6 + a_3 \theta^3 + a_0)$, and so $\sqrt{r}=a_9 \theta^9 + a_6 \theta^6 + a_3 \theta^3 + a_0$. 

Similar to Lemma \ref{squarelem}, it follows that either $a_0=a_6=0$ or $a_3=a_9=0$, which implies that one of $r(a^2-4b)$, $r(-a+2\sqrt{b})$ or $r(-a-2\sqrt{b})$ is in $\Q^2$. This proves the necessity. The sufficiency follows from construction.
\end{proof}

We now rule out the possibility for certain pairs of $(G_4,G_6)$. We then show that each of the remaining pairs of $(G_4,G_6)$ is possible and determine the exact possible Galois groups for each such pair. We then give numerical examples for each of the possible Galois groups.

\begin{lemma}
Let $f(x)=x^{12}+ax^6+b \in \Q[x]$ be an irreducible polynomial. Then $$(G_4,G_6) \notin \{ (4T1,6T1), (4T1,6T2), (4T1,6T5) \}.$$
\end{lemma}
\begin{proof} 
By Propositions \ref{g4} and \ref{g6}, this is equivalent to showing that we cannot have both $b(a^2-4b) \in \Q^2$ and $3(4b-a^2) \in \Q^2$. Suppose that this is not the case. Then their product is $-3b(a^2-4b)^2 \in \Q^2$, which implies that $-3b \in \Q^2$ and hence, $-\big(4(-3b)+a^2\big)=3(4b-a^2) \in \Q^2$ implies that $-3b=a=0$, a contradiction.
\end{proof}

\begin{proposition} \label{dodecicgal}
Let $f(x)=x^{12}+ax^6+b \in \Q[x]$ be an irreducible polynomial.
\begin{enumerate}
\item If $G_4$ is $4T1$ and
\begin{enumerate}
\item if $G_6$ is $6T3$, then $\Gal(f)$ is $12T11$.
\item if $G_6$ is $6T9$, then $\Gal(f)$ is $12T39$.
\end{enumerate}
\item If $G_4$ is $4T2$ and
\begin{enumerate}
\item if $G_6$ is $6T1$, then $\Gal(f)$ is $12T2$.
\item if $G_6$ is $6T2$, then $\Gal(f)$ is $12T3$.
\item if $G_6$ is $6T5$, then $\Gal(f)$ is $12T18$.
\item if $G_6$ is $6T3$, then $\Gal(f)$ is either $12T3$ or $12T10$.
\item if $G_6$ is $6T9$, then $\Gal(f)$ is either $12T16$ or $12T37$.
\end{enumerate}
\item If $G_4$ is $4T3$ and
\begin{enumerate}
\item if $G_6$ is $6T1$, then $\Gal(f)$ is $12T14$.
\item if $G_6$ is $6T2$, then $\Gal(f)$ is $12T15$.
\item if $G_6$ is $6T5$, then $\Gal(f)$ is $12T42$.
\item if $G_6$ is $6T3$, then $\Gal(f)$ is $12T12$, $12T13$ or $12T28$.
\item if $G_6$ is $6T9$, then $\Gal(f)$ is either $12T38$ or $12T81$.
\end{enumerate}
\end{enumerate}
Moreover, each of these possibilities does occur.
\end{proposition}

\begin{proof}
We examine all 301 conjugacy classes of transitive subgroups of $S_{12}$. In view of Lemma \ref{lem1}, for each possible pair $(G_4,G_6)$ we filter a list of possible Galois groups $\Gal(f)$ in Table \ref{tabley1} based on the following criteria:
\begin{itemize} \itemsep0em
\item Have both $G_4$ and $G_6$ in the subfield content of $f(x)$,
\item Is contained in $A_{12}$ if $G_4$ is $4T2$, is not contained in $A_{12}$ otherwise,
\item Have order at most $18|G_4|$ and $4|G_6|$.
\end{itemize}
The numerical examples in Table \ref{tabley2} show that each of these possibilities does occur.
\end{proof}

\begin{table}[h!]
\caption{Possible Galois groups (by $T$ number) of irreducible polynomials $f(x)=x^{12}+ax^6+b \in \Q[x]$ based on the Galois group $G_4$ and $G_6$ of $x^4+ax^2+b$ and $x^6+ax^3+b$, respectively. \label{tabley1}}
\begin{center}
{\begin{tabular}{|c|C{1.5cm}|C{1.5cm}|C{1.5cm}|C{1.5cm}|C{1.5cm}|} \hline
\diagbox{$\mathbf{G_4}$}{$\mathbf{G_6}$} & $6T1$ & $6T2$ & $6T5$ & $6T3$ & $6T9$ \\ \hline
$4T1$ &&&& 11 & 39 \\ \hline
$4T2$ & 2 & 3 & 18 & 3,10 & 16,37 \\ \hline
$4T3$ & 14 & 15 & 42 & 12,13,28 & 38,81 \\ \hline
\end{tabular}}
\end{center}
\end{table}

\begin{table}[h!]
\caption{Numerical examples of irreducible polynomials $x^{12}+ax^6+b \in \Q[x]$ with Galois group $G_{12}$, where $x^4+ax^2+b$ and $x^6+ax^3+b$ have Galois groups $G_4$ and $G_6$, respectively. \label{tabley2}}
\begin{center}
{\begin{tabular}{|c|c|c|c|} \hline
$\mathbf{G_4}$ & $\mathbf{G_6}$ & $\mathbf{G_{12}}$ & \textbf{Polynomial} \\ \hline
$4T1$ & $6T3$ & $12T11$ & $x^{12}+8x^6+8$ \\ \hline
$4T1$ & $6T9$ & $12T39$ & $x^{12}+4x^6+2$ \\ \hline
$4T2$ & $6T1$ & $12T2$ & $x^{12}-x^6+1$ \\ \hline
$4T2$ & $6T2$ & $12T3$ & $x^{12}+572x^6+470596$ \\ \hline
$4T2$ & $6T5$ & $12T18$ & $x^{12}+2x^6+4$ \\ \hline
$4T2$ & $6T3$ & $12T3$ & $x^{12}+5x^6+1$ \\ \hline
$4T2$ & $6T3$ & $12T10$ & $x^{12}+3x^6+1$ \\ \hline
$4T2$ & $6T9$ & $12T16$ & $x^{12}-x^6+4$ \\ \hline
$4T2$ & $6T9$ & $12T37$ & $x^{12}+x^6+4$ \\ \hline
$4T3$ & $6T1$ & $12T14$ & $x^{12}+9x^6+27$ \\ \hline
$4T3$ & $6T2$ & $12T15$ & $x^{12}+3$ \\ \hline
$4T3$ & $6T5$ & $12T42$ & $x^{12}+x^6+7$ \\ \hline
$4T3$ & $6T3$ & $12T12$ & $x^{12}+x^6-27$ \\ \hline
$4T3$ & $6T3$ & $12T13$ & $x^{12}-3$ \\ \hline
$4T3$ & $6T3$ & $12T28$ & $x^{12}+2$ \\ \hline
$4T3$ & $6T9$ & $12T38$ & $x^{12}+4x^6-2$ \\ \hline
$4T3$ & $6T9$ & $12T81$ & $x^{12}+x^6+2$ \\ \hline
\end{tabular}}
\end{center}
\end{table}

We remark that if $G_4$ is $4T1$ or $G_6 \in \{6T1,6T2,6T5\}$, then $\Gal(f)$ can be identified uniquely, whereas there are more than one possible $\Gal(f)$ when $G_4 \in \{4T2,4T3\}$ and $G_6 \in \{6T3,6T9\}$.

\subsection{Classification} \label{G12b}

If $G_4 \in \{4T2,4T3\}$ and $G_6 \in \{6T3,6T9\}$, most of the possible $\Gal(f)$ can be distinguished by their orders (see Table \ref{tabley3}). This motivates us to calculate the degree of the splitting field of $f(x)$ over $\Q$, and gives us the following.

\begin{table}[h!] 
\caption{Possible Galois groups $G_{12}$ of irreducible polynomials $x^{12}+ax^6+b$ and their orders, where $x^4+ax^2+b$ and $x^6+ax^3+b$ have Galois groups $G_4 \in \{4T2,4T3\}$ and $G_6 \in \{6T3,6T9\}$, respectively. \label{tabley3}}
\begin{center} \small
{\begin{tabular}{|c|c|c|c|c|c|c|c|c|c|} \hline
$\mathbf{(G_4,G_6)}$ & \multicolumn{2}{c|}{$(4T2,6T3)$} & \multicolumn{2}{c|}{$(4T2,6T9)$} & \multicolumn{3}{c|}{$(4T3,6T3)$} & \multicolumn{2}{c|}{$(4T3,6T9)$} \\ \hline
$\mathbf{G_{12}}$ & $12T3$ & $12T10$ & $12T16$ & $12T37$ & $12T12$ & $12T13$ & $12T28$ & $12T38$ & $12T81$ \\ \hline
$\textbf{Order}$ & 12 & 24 & 36 & 72 & 24 & 24 & 48 & 72 & 144 \\ \hline
\end{tabular}}
\end{center}
\end{table}

\begin{proposition} \label{dodecic1}
Let $f(x)=x^{12}+ax^6+b \in \Q[x]$ be an irreducible polynomial.
\begin{enumerate}
\item If $(G_4,G_6)=(4T2,6T3)$, then $\Gal(f)$ is
\begin{enumerate}
\item $12T3$ if $3(a+2\sqrt{b}) \in \Q^2$ or $3(a-2\sqrt{b}) \in \Q^2$.
\item $12T10$ otherwise.
\end{enumerate}
\item If $(G_4,G_6)=(4T2,6T9)$, then $\Gal(f)$ is
\begin{enumerate}
\item $12T16$ if $3(a+2\sqrt{b}) \in \Q^2$ or $3(a-2\sqrt{b}) \in \Q^2$.
\item $12T37$ otherwise.
\end{enumerate}
\item If $(G_4,G_6)=(4T3,6T3)$, then $\Gal(f)$ is
\begin{enumerate}
\item either $12T12$ or $12T13$ if $-3b \in \Q^2$ or $3b(4b-a^2) \in \Q^2$.
\item $12T28$ otherwise.
\end{enumerate}
\item If $(G_4,G_6)=(4T3,6T9)$, then $\Gal(f)$ is
\begin{enumerate}
\item $12T38$ if either $-3b \in \Q^2$ or $3b(4b-a^2) \in \Q^2$.
\item $12T81$ otherwise.
\end{enumerate}
\end{enumerate}
\end{proposition}

\begin{proof}
 Let $\theta$ be a root of $f(x)$. Factoring $f(x)$ over $\Q(\theta)$, we have 
\begin{eqnarray*}
f(x)=(x-\theta)(x+\theta)\left(x-\frac{\sqrt[6]{b}}{\theta}\right)\left(x+\frac{\sqrt[6]{b}}{\theta}\right)f_1(x)f_1(-x)f_2(x)f_2(-x)
\end{eqnarray*}
where
\begin{eqnarray*}
f_1(x) = x^2+\theta x+\theta^2, \quad f_2(x)=x^2+\frac{\sqrt[6]{b}}{\theta}x+\frac{\sqrt[3]{b}}{\theta^2}.
\end{eqnarray*}
It is then easy to verify that the roots of $f(x)$ are
\begin{eqnarray*}
\pm\theta, \ \pm\theta\omega, \ \pm\theta\omega^2, \ \pm\frac{\sqrt[6]{b}}{\theta}, \ \pm\frac{\sqrt[6]{b}}{\theta}\omega, \ \pm\frac{\sqrt[6]{b}}{\theta}\omega^2,
\end{eqnarray*}
where $\omega=(-1+\sqrt{-3})/2$ is a primitive third root of unity. It follows that the splitting field of $f(x)$ is $\Q(\theta,\sqrt{-3},\sqrt[6]{b})=\Q(\theta,\sqrt{-3},\sqrt{b},\sqrt[3]{b})$. Now let $K=\Q(\theta)$, $K'=K(\sqrt{-3},\sqrt{b})$ and $L=K'(\sqrt[3]{b})$, so $L$ is the splitting field of $f(x)$. We have
\begin{eqnarray*}
[L:\Q]=[L:K'][K':K][K:\Q]=12[L:K'][K':K].
\end{eqnarray*}
We note that $K'$ is a biquadratic extension of $K$, and it follows that
\begin{eqnarray*}
[K':K] = 
\begin{cases}
1, & \text{if both $-3$ and $b$ are in $\Q(\theta)^2$} \\
2, & \text{if exactly one of $-3$, $b$ or $-3b$ is in $\Q(\theta)^2$} \\
4, & \text{if none of $-3$, $b$ and $-3b$ are in $\Q(\theta)^2$}
\end{cases}.
\end{eqnarray*}
We also note that $[L:K'] \in \{1,3\}$. If $G_6$ is $6T3$, then either $b \in \Q^3$ or $r(x)=x^3-3bx+ab$ is reducible. For the former subcase, clearly $\sqrt[3]{b} \in \Q \subset K'$. For the latter subcase, Proposition \ref{g6} implies that $r(x)$ has a rational root $r$, so we may write $a=(3br-r^3)/b$. Then it can be verified that
\begin{eqnarray*}
\left( \frac{r}{b-r^2} \theta^{10} +  \frac{-b^2+3br^2-r^4}{b(b-r^2)} \theta^4 \right)^3 = b.
\end{eqnarray*}
It follows that $\sqrt[3]{b} \in K \subset K'$. As such, for either subcase we have $[L:K']=1$. 

Now if $G_6$ is $6T9$, we claim that $[L:K']=3$. Suppose to the contrary that $[L:K']=1$. Note that the orders of possible $\Gal(f)$ are multiples of 36, but $[K':K]$ is not a multiple of three, a contradiction. It follows that
\begin{eqnarray*}
[L:K'] = 
\begin{cases}
1, & \text{if $G_6$ is $6T3$} \\
3, & \text{if $G_6$ is $6T9$}
\end{cases}.
\end{eqnarray*}

We first consider the case where $G_4$ is $4T2$. If $G_6$ is $6T3$, then $\Gal(f)$ is $12T3$ if and only if $[K':K]=1$. This occurs if and only if $-3 \in \Q(\theta)^2$, since $b \in \Q^2 \subset \Q(\theta)^2$. Likewise if $G_6$ is $6T9$, then $\Gal(f)$ is $12T16$ if and only if $[K':K]=1$ if and only if $-3 \in \Q(\theta)^2$. By Proposition \ref{prop2}, $-3 \in \Q(\theta)^2$ if and only if $3(4b-a^2)$, $3(a+2\sqrt{b})$ or $3(a-2\sqrt{b})$ is in $\Q^2$, but $3(4b-a^2) \notin \Q^2$ by Proposition \ref{g6}. This proves $(1)$ and $(2)$.
\par We now consider the case where $G_4$ is $4T3$. If $G_6$ is $6T3$, then $\Gal(f)$ is either $12T12$ or $12T13$ if and only if $[K':K]=2$. By Proposition \ref{prop2}, $-3 \in \Q(\theta)^2$ if and only if $3(4b-a^2) \in \Q^2$, whereas $b \in \Q(\theta)^2$ if and only if $b \in \Q^2$ or $b(a^2-4b) \in \Q^2$, but none of them are in $\Q^2$ by Propositions \ref{g4} and \ref{g6}. Hence, $[K':K]=2$ if and only if $-3b \in \Q(\theta)^2$. Likewise if $G_6$ is $6T9$, then $\Gal(f)$ is $12T38$ if and only if $[K':K]=2$ if and only if $-3b \in \Q(\theta)^2$. By Proposition \ref{prop2}, $-3b \in \Q(\theta)^2$ if and only if $-3b \in \Q^2$ or $3b(4b-a^2) \in \Q^2$. This proves $(3)$ and $(4)$.
\end{proof}

%\begin{proof}
%***We first consider the case where $G_4$ is $4T2$. If $G_6$ is $6T3$, then $\Gal(f)$ is 12T3 (resp. 12T10) if $[K':K]=1$ (resp. 2). Since $b \in \Q^2 \subset \Q(\theta)^2$, it follows that $[K':K]=1$ (resp. 2) if $-3 \in \Q(\theta)^2$ (resp. $-3 \notin \Q(\theta)^2$). Likewise if $G_6$ is $6T9$, then $\Gal(f)$ is 12T16 (resp. 12T37) if $[K':K]=1$ (resp. 2), and it follows that $[K':K]=1$ (resp. 2) if $-3 \in \Q(\theta)^2$ (resp. $-3 \notin \Q(\theta)^2$). By Proposition \ref{prop2}, $-3 \in \Q(\theta)^2$ if and only if any of $3(4b-a^2)$, $3(a+2\sqrt{b})$ or $3(a-2\sqrt{b})$ is in $\Q^2$. However, we have $3(4b-a^2) \notin \Q^2$ by Proposition \ref{g6}. This proves $(1)$ and $(2)$.
%\par We now consider the case where $G_4$ is $4T3$. If $G_6$ is $6T3$, then $\Gal(f)$ is either 12T12 or 12T13 (resp. 12T28) if $[K':K]=2$ (resp. 4). By Proposition \ref{prop2}, $-3 \in \Q(\theta)^2$ if and only if $3(4b-a^2) \in \Q^2$, whereas $b \in \Q(\theta)^2$ if and only if $b \in \Q^2$ or $b(a^2-4b) \in \Q^2$. It follows from Propositions \ref{g4} and \ref{g6} that each of $3(4b-a^2)$, $b$ and $b(a^2-4b)$ is not in $\Q^2$ and hence, $[K':K]=2$ (resp. 4) if $-3b \in \Q(\theta)^2$ (resp. $-3 \notin \Q(\theta)^2$). Likewise if $G_6$ is $6T9$, then $\Gal(f)$ is 12T38 (resp. 12T81) if $[K':K]=2$ (resp. 4), and it follows that $[K':K]=2$ (resp. 4) if $-3b \in \Q(\theta)^2$ (resp. $-3b \notin \Q(\theta)^2$). By Proposition \ref{prop2}, $-3b \in \Q(\theta)^2$ if and only if $-3b \in \Q^2$ or $3b(4b-a^2) \in \Q^2$. This proves $(3)$ and $(4)$.
%\end{proof}

To complete our classification, we need to distinguish between $12T12$ and $12T13$. 

\begin{proposition} \label{dodecic2}
Let $f(x)=x^{12}+ax^6+b \in \Q[x]$ be an irreducible polynomial with $(G_4,G_6)=(4T3,6T3)$ and either $-3b \in \Q^2$ or $3b(4b-a^2) \in \Q^2$.
\begin{enumerate}
\item If $r(x)$ is reducible, then $\Gal(f)$ is
\begin{enumerate}
\item $12T12$ if $3b(4b-a^2) \in \Q^2$.
\item $12T13$ if $-3b \in \Q^2$.
\end{enumerate}
\item If $b \in \Q^3$, then $\Gal(f)$ is
\begin{enumerate}
\item $12T12$ if $-3b \in \Q^2$.
\item $12T13$ if $3b(4b-a^2) \in \Q^2$.
\end{enumerate}
\end{enumerate}
\end{proposition}

\begin{proof}
We use the relation (\ref{resultant}) to compute the linear resolvent $R(x)$ of $f(x)$ corresponding to $x_1+x_2$. The list of degrees of irreducible factors of $R(x)$ are $6,12^5$ and $6,12^3,24$ for $12T12$ and $12T13$, respectively. Now if $(G_4,G_6)=(4T3,6T3)$ then either $r(x)=x^3-3bx+ab$ is reducible or $b \in \Q^3$. For either case, we have $R(x)$ is the product of $x^6$ with three irreducible degree twelve polynomials namely, $f(x)=x^{12}+ax^6+b$, $R_1(x^6)=x^{12}-27ax^6+729b$, and $S(x^2)$, and a degree 24 polynomial $S_1(x)$, with $S_1(x)$ reducible if and only if $\Gal(f)$ is $12T12$.
\par We determine the three length twelve orbits for the action of $G \in \{12T12,12T13\}$ on the 2-sets of twelve letters. For each of these orbits $O$, we can determine the image of the permutation representation of $G$ acting on $O$. For $G=12T12$, each of these three orbits corresponds to $12T12$ whereas for $G=12T13$, two orbits correspond to $12T13$ and the third orbit corresponds to $12T15$. Suppose to the contrary that $R_1(x^6)$ has Galois group $12T15$. Then $R_1(x^3)$ has Galois group $6T2$, contradicting Proposition \ref{g6}. It follows that if $G$ is $12T13$, then the Galois group of $S(x^2)$ is $12T15$. In each case, the subfield content of $S(x^2)$ contains a unique transitive subgroup $H_6$ of $S_6$. Since $S(x)$ defines a quadratic subfield of the field defined by $S(x^2)$, it follows that the Galois group of $S(x)$ is $H_6$. In particular, we have $H_6=6T3$ if $\Gal(f)=12T12$ and $H_6=6T_2$ if $\Gal(f)=12T13$.

If $r(x)$ is reducible, then $r(x)$ has a rational root $r$, so we may write $a=(3br-r^3)/b$. For such case, $S(x)=x^6+Ax^3+B$ where 
\begin{eqnarray*}
A = \frac{-2r(r^2+12b)}{b} \quad \text{ and } \quad B = \frac{(r^2-4b)^3}{b^2}.
\end{eqnarray*}
By Proposition \ref{g6}, $\Gal(S)$ is $6T2$ if $3(4B-A^2) \in \Q^2$ and $6T3$ otherwise. Now
\begin{eqnarray*}
3(4B-A^2)=\frac{-48(3r^2+4b)^2}{b}=-3b\left(\frac{4(3r^2+4b)}{b}\right)^2,
\end{eqnarray*}
and so $3(4B-A^2) \in \Q^2$ if and only if $-3b \in \Q^2$. This proves $(1)$.
\par Now if $b \in \Q^3$, let $\beta \in \Q$ be the principal cube root of $b$. For such case,
\begin{eqnarray*}
\begin{split}
S_1(x) &= x^{24} + 18 \beta x^{20} + 4 a x^{18} + 267 \beta^2 x^{16} + 18 a \beta x^{14} + (6 a^2 + 1018 b) x^{12} \\ &- 762 a \beta^2 x^{10} + (-18 a^2 \beta + 3177 b\beta) x^8 + (4 a^3 - 1042 a b) x^6 \\ &+ (267 a^2 \beta^2 + 228 b\beta^2) x^4 + (-18 a^3 \beta + 72 a b\beta) x^2 + a^4 - 8 a^2 b + 16 b^2.
\end{split}
\end{eqnarray*}
 Now if $-3b \in \Q^2$, then $\beta=-3q^2$ for some $q \in \Q$ and $S_1(x)$ factors as the product of $S_0(q)$ and $S_0(-q)$, where
\begin{eqnarray*}
\begin{split}
S_0(t) &= x^{12} + 18 t x^{10} + 135 t^2 x^8 + (2 a + 486 t^3) x^6 \\ &+ (18 a t + 837 t^4) x^4 + (27 a t^2 + 486 t^5) x^2 + 
a^2 + 108 t^6.
\end{split}
\end{eqnarray*}
This proves $(2)(a)$. We also have
\begin{eqnarray*}
S(x)=x^6-18\beta x^4+2a x^3+57\beta^2 x^2+18a\beta x+a^2-4 b.
\end{eqnarray*}
We use the relation (\ref{resultant2}) to compute the linear resolvent $\widetilde{R}(x)$ of $S(x)$ corresponding to $x_1 x_2$. The list of degrees of irreducible factors of $S(x)$ are $3^3,6$ and $3,6^2$ for $\Gal(S)=6T2$ and $\Gal(S)=6T3$, respectively. $\widetilde{R}(x)$ is the product of $x^3+6\beta x^2+9\beta^2x+4b-a^2$ with two degree six polynomials $\widetilde{R}_1(x)$ and $\widetilde{R}_2(x)$ where
\begin{eqnarray*}
\begin{split}
\widetilde{R}_1(x) &= x^6 - 18 \beta x^5 + 105 \beta^2 x^4 + (-2 a^2 - 224 b) x^3 + (-18 a^2 \beta + 216 b\beta) x^2 \\ &+ (24 a^2 \beta^2 - 96 b\beta^2) x + a^4 - 8 a^2 b + 16 b^2,
\end{split}
\end{eqnarray*}
\begin{eqnarray*}
\begin{split}
\widetilde{R}_2(x) &= x^6 + 30 \beta x^5 + 297 \beta^2 x^4 + (-2 a^2 + 1088 b) x^3 +  (78 a^2 \beta + 984 b\beta) x^2 \\ &+  (-72 a^2 \beta^2 + 288 b\beta^2) x + a^4 - 8 a^2 b + 16 b^2.
\end{split}
\end{eqnarray*}
Now if $3b(4b-a^2) \in \Q^2$, then $4b-a^2=3bq^2$ for some $q \in \Q$, and so $b=a^2/(4-3q^2)$. Let $v=a(4-3q^2)$ so that $b=v^2/(4-3q^2)^3$. Since $b \in \Q^3$, it follows that $v \in \Q^3$ and hence, $v=u^3$ for some $u \in \Q$. Therefore, $a = u^3/(4 - 3 q^2)$ and $\beta = u^2/(4 - 3 q^2)$. Then $\widetilde{R}_2(x)$ factors as the product of $\widetilde{R}_0(q)$ and $\widetilde{R}_0(-q)$, where
\begin{eqnarray*}
\widetilde{R}_0(t)=x^3+\frac{15u^2}{4-3t^2}x^2+\frac{18(t+2)u^4}{(4-3t^2)^2}x+\frac{3t^2u^6}{(4-3t^2)^3}.
\end{eqnarray*}
This proves $(2)(b)$.
\end{proof}

\subsection{Summary} \label{G12d}

By combining Propositions \ref{g4}, \ref{g6}, \ref{dodecicgal}, \ref{dodecic1} and \ref{dodecic2}, we have the following algorithm that provides an elementary characterization for each of the sixteen possible Galois groups of $x^{12}+ax^6+b$.

\begin{theorem} \label{mthm} 
Let $f(x)=x^{12}+ax^6+b \in \Q[x]$ be an irreducible polynomial and $r(x)=x^3-3bx+ab$. Then the following algorithm returns $\Gal(f)$.
\begin{enumerate} \itemsep0em
\item If $b(a^2-4b) \in \Q^2$, then
\begin{enumerate} \itemsep0em
\item If $b \in \Q^3$ or $r(x)$ is reducible, return $12T11$ and terminate.
\item Otherwise, return $12T39$ and terminate.
\end{enumerate}
\item Else if $b \in \Q^2$, then
\begin{enumerate}\itemsep0em
\item If $3(4b-a^2) \in \Q^2$, then
\begin{enumerate} \itemsep0em
\item If $r(x)$ is reducible, return $12T3$ and terminate.
\item Else if $b \in \Q^3$, return $12T2$ and terminate.
\item Otherwise, return $12T18$ and terminate.
\end{enumerate}
\item Else if $3(4b-a^2) \notin \Q^2$ and either $3(a+2\sqrt{b}) \in \Q^2$ or $3(a-2\sqrt{b}) \in \Q^2$, then
\begin{enumerate} \itemsep0em
\item If $b \in \Q^3$ or $r(x)$ is reducible, return $12T3$ and terminate.
\item Otherwise, return $12T16$ and terminate.
\end{enumerate}
\item Else if $3(4b-a^2) \notin \Q^2$, $3(a+2\sqrt{b}) \notin \Q^2$ and $3(a-2\sqrt{b}) \notin \Q^2$, then
\begin{enumerate} \itemsep0em
\item If $b \in \Q^3$ or $r(x)$ is reducible, return $12T10$ and terminate.
\item Otherwise, return $12T37$ and terminate.
\end{enumerate}
\end{enumerate}
\item Else if $b(a^2-4b) \notin \Q^2$ and $b \notin \Q^2$, then
\begin{enumerate} \itemsep0em
\item If $3(4b-a^2) \in \Q^2$, then
\begin{enumerate} \itemsep0em
\item If $r(x)$ is reducible, return $12T15$ and terminate.
\item Else if $b \in \Q^3$, return $12T14$ and terminate.
\item Otherwise, return $12T42$ and terminate.
\end{enumerate}
\item Else if $3(4b-a^2) \notin \Q^2$ and either $-3b \in \Q^2$ or $3b(4b-a^2) \in \Q^2$, then
\begin{enumerate} \itemsep0em
\item If $b \in \Q^3$, then
\begin{enumerate} \itemsep0em
\item If $-3b \in \Q^2$, return $12T12$ and terminate.
\item Otherwise, return $12T13$ and terminate.
\end{enumerate}
\item Else if $b \notin \Q^3$ and $r(x)$ is reducible, then
\begin{enumerate} \itemsep0em
\item If $3b(4b-a^2) \in \Q^2$, return $12T12$ and terminate.
\item Otherwise, return $12T13$ and terminate.
\end{enumerate}
\item Otherwise, return $12T38$ and terminate.
\end{enumerate}
\item Else if $3(4b-a^2) \notin \Q^2$, $-3b \notin \Q^2$ and $3b(4b-a^2) \notin \Q^2$, then
\begin{enumerate} \itemsep0em
\item If $b \in \Q^3$ or $r(x)$ is reducible, return $12T28$ and terminate.
\item Otherwise, return $12T81$ and terminate.
\end{enumerate}
\end{enumerate}
\end{enumerate}
\end{theorem}

\pagebreak 
\section*{Acknowledgements}

The author thanks both referees for their careful reading of the manuscript and for their many helpful comments and suggestions.

\end{document}